\documentclass[bibalpha]{amsart}
\usepackage{verbatim}
\usepackage{enumerate}

\title{Defining coarsenings of valuations}

\author{Franziska Jahnke} 
\address{Institut f\"ur Mathematische Logik\\Einsteinstr. 62\\48149 M\"unster, 
Germany}
\email{franziska.jahnke@uni-muenster.de}

\author{Jochen Koenigsmann}
\address{Mathematical Institute\\Radcliffe Observatory Quarter\\
Woodstock Road\\Oxford OX2 6GG, UK}
\email{koenigsmann@maths.ox.ac.uk}

\usepackage{enumerate}
\usepackage{braket}
\usepackage{verbatim}
\usepackage{txfonts}
\newtheorem{Th}{Theorem}[section]
\newtheorem{Thm}[Th]{Theorem}

\newtheorem*{Def}{Definition}
\newtheorem{Cor}[Th]{Corollary}
\newtheorem{Prop}[Th]{Proposition}

\newtheorem{Ex}[Th]{Example}
\newtheorem*{Rem}{Remark}

\newtheorem{Lem}[Th]{Lemma}
\newtheorem{Q}[Th]{Question}

\newtheorem*{Lem*}{Lemma}

\begin{document}
\begin{abstract}We 
study the question
which henselian fields admit definable henselian valuations (with or 
without parameters). 
We show that every field which admits a henselian valuation with non-divisible
value group admits a parameter-definable (non-trivial) henselian valuation.
In equicharacteristic $0$, we give a complete characterization of
henselian fields admitting a parameter-definable (non-trivial) 
henselian valuation.
We also obtain partial characterization results of fields admitting
 $\emptyset$-definable (non-trivial) henselian valuations. 
We then draw some Galois-theoretic conclusions from our results.  
\end{abstract}

\maketitle

\section{Introduction}
We study the question which henselian 
fields admit non-trivial henselian valuations
which are \emph{definable},
i.e., for which the valuation ring is first-order definable in
the language of rings.
Furthermore, we investigate whether
parameters are required for these definitions. Here, we call a field \emph{henselian}
if it admits some non-trivial henselian valuation.
There has been considerable progress in the area of definable henselian 
valuations over the last few years. 
Most recent results are focussed on defining a 
specific \emph{given}
henselian valuation on a henselian field, sometimes with formulae of 
low quantifier complexity (see \cite{CDLM}, \cite{Hon14}, 
\cite{AK14}, \cite{Fe14}, \cite{JK14a}, \cite{Pr14} and \cite{FJ14}). 
The question considered in this paper is however 
whether 
a given
henselian field admits at least \emph{some} non-trivial definable henselian
valuation. 
There are many henselian fields having both definable and non-definable
henselian valuations (cf.\,Example \ref{two}).

Neither separably closed fields nor real
closed fields admit any non-trivial definable valuations.
For real closed fields, this follows from quantifier elimination in the language
of ordered rings $\mathcal{L}_{ring} \cup \{<\}$: Any definable subset of a real
closed field is a finite union of intervals and points, and in particular
not a valuation ring. The fact that separably closed fields
do not admit any definable valuations is explained in 
\cite[Introduction, p.~1]{Koe94}.  
Hence, we focus on henselian fields which are neither separably closed 
nor real closed.
Any such field $K$ interprets a finite Galois extension $F$ such that for some 
prime $p$,
the canonical $p$-henselian valuation $v_F^p$ is $\emptyset$-definable
and non-trivial (cf.~Section \ref{sec:can} for the definition of the canonical 
$p$-henselian valuation).
This valuation is in particular 
comparable to any henselian valuation on $F$.
If $v_F^p$ is already henselian, then its restriction to $K$ gives a non-trivial
definable henselian valuation on $K$. If $v_F^p$ is non-henselian,
then any henselian valuation on $F$ is a coarsening of $v_F^p$.
Thus, the task of finding definable henselian valuations on $F$ (and thus on
$K$) comes down to defining (henselian) coarsenings of $v_F^p$.

We use two different methods to define coarsenings of a given (definable)
valuation on a field $F$:
In Section \ref{sec:anti}, we introduce \emph{$p$-antiregular} ordered
abelian groups. The case distinction between $p$-antiregular and
non-$p$-antiregular value groups is a key step in several of our
proofs. We also show how to define, 
for any prime $p$, the maximal $p$-divisible quotient of an ordered
abelian group (without any parameters).
The construction should be well-known to anyone with a good knowlegde
of definable convex subgroups of ordered abelian groups.
However, our approach is rather short and self-contained and should
be easily accessible to anyone with an interest in valuation theory. 
The main result of the section is Proposition 
\ref{antireg}, which gives conditions on the value group of a henselian
valuation under which some non-trivial coarsening is 
$\emptyset$-definable. 
In this section, we also discuss the construction of a field
which will be helpful in examples and counterexamples at several points 
later on (see Example \ref{PZ}). 
 
The other method we use is introduced in Section 
\ref{sec:AEJ}. 
Here, we discuss a certain class of parameter-definable convex subgroups
of ordered abelian groups. Again, our treatment is rather short and
self-contained.
This gives us the means to find a definable
henselian valuation on $K$ whenever some henselian valuation on $K$ has a
non-divisible value group (Proposition \ref{nondiv}).

We then proceed to apply these two basic constructions to give criteria
for the existence of $\emptyset$-definable and definable henselian
valuations. These criteria are phrased in terms of the value group 
$v_KK$ and the residue field $Kv_K$ of the canonical
henselian valuation $v_K$ on $K$
(cf.~Section \ref{sec:can} for the definition of $v_K$). 

In Section \ref{sec:0}, 
we discuss the existence of a non-trivial $\emptyset$-definable
henselian valuation on a field $K$. Here, our main result is the following:
{
\renewcommand{\theTh}{A}
\begin{Thm} \label{1}
Let $K \neq K^\mathrm{sep}$ be a henselian field. Then $K$ admits a
$\emptyset$-definable non-trivial henselian valuation unless
\begin{enumerate}
\item $Kv_K \neq Kv_K^\mathrm{sep}$ \emph{and}
\item $Kv_K$ is $t$-henselian \emph{and}
\item $v_KK$ is $p$-antiregular for all 
primes $p$ with $v_KK \neq p v_KK$ (e.g., if $v_KK$ is divisible). 
\end{enumerate}
\end{Thm}
}
See Section \ref{sec:anti} for the definition of $t$-henselianity.
Note that the case that $K$ is real closed is covered by the `unless' setting:
In this case, $Kv_K$ is an archimedean ordered real closed field and
hence $t$-henselian without being henselian.
The 
theorem implies that
every (non-separably or non-real closed) henselian field of 
finite transcendence degree over its prime field
admits a non-trivial $\emptyset$-definable henselian valuation 
(Corollary \ref{cor:fin}). 
As another consequence, we get a classification of all fields with 
small absolute
Galois group admitting $\emptyset$-definable henselian valuations,
provided that the canonical henselian valuation has
residue characteristic $0$ (Corollary \ref{cor:small}). However,
the conditions described in Theorem \ref{1} are not sufficient for a full
characterization of fields admitting $\emptyset$-definable
non-trivial henselian valuations (see Example \ref{ex:anti} and Proposition
\ref{1.1}).

In Section \ref{sec:1}, we discuss the existence of a non-trivial definable 
henselian valuation on a field $K$. Here, we prove the following:
{
\renewcommand{\theTh}{B}
\begin{Thm} \label{2}
Let $K \neq K^\mathrm{sep}$ be a henselian field. Then $K$ admits a
definable non-trivial henselian valuation (using at most $1$ parameter) 
unless
\begin{enumerate}
\item $Kv_K \neq Kv_K^\mathrm{sep}$ \emph{and}
\item $Kv_K$ is $t$-henselian \emph{and}
\item $v_KK$ is divisible. 
\end{enumerate}
\end{Thm}
}
Furthermore, in equicharacteristic $0$, this Theorem gives rise 
to a characterization of henselian fields admitting non-trivial
definable henselian valuations (cf.~Corollary \ref{0,0}). 
We also give an example of a henselian field without a 
definable non-trivial henselian valuation
and an example of a henselian field which admits a definable non-trivial
henselian valuation but no $\emptyset$-definable such.

We study the existence of ($\emptyset$-)definable
($p$-)henselian valuations tamely branching at $p$ in the last section 
(which also
contains the definition of tamely branching valuations). By the results in
\cite{Koe03}, these
are exactly the henselian valuations encoded in the absolute Galois group
$G_K$ of a field $K$. Our main result in this context is as follows:
{
\renewcommand{\theTh}{C}
\begin{Thm} \label{3}
Let $K$ be a field and $p$ a prime. 
\begin{enumerate}
\item If $K$ admits a henselian valuation $v$ tamely branching at $p$,
then $K$ admits a definable such (using at most $1$ parameter). 
\item Assume $\zeta_p \in K$ and, in case
$p=2$ and $\mathrm{char}(K)=0$, assume also $\sqrt{-1} \in K$.
If $K$
admits a $p$-henselian valuation tamely branching at $p$,
then $K$ admits a $\emptyset$-definable such.
\end{enumerate}
\end{Thm}
This theorem is an immediate consequence of Propositions 
\ref{ptame} and \ref{tb}.  
As an application, we also obtain some Galois-theoretic consequences 
(cf.~Corollaries \ref{cor:tame} and \ref{cor:nottame}).
}

\section{Canonical ($p$-)henselian valuations}
Throughout the paper, we use the following notation: \label{sec:can}
For a valued field $(K,v)$, we write $Kv$ for its residue field and
$vK$ for its value group. Furthermore, we denote the
valuation ring of $v$ by 
$\mathcal{O}_v$ and its maximal ideal by $\mathfrak{m}_v$.
If $p \neq \mathrm{char}(K)$ is a prime, we write $\zeta_p \in K$ 
to denote that $K$ contains a primitive $p$th root of unity.
For basic facts about ($p$-)henselian valued fields, we refer the reader
to \cite{EP05}.

\subsection{The canonical henselian valuation}
Let $K$ be a \emph{henselian} field, i.e., 
assume that $K$ admits some non-trivial
henselian valuation. In general, $K$ may admit many
non-trivial henselian valuations,
however, unless $K$ is separably closed, 
they all induce the same topology on $K$. When we ask which henselian fields
admit a definable non-trivial 
henselian valuation, we do not specify which one should 
be definable.
In all our constructions, we define coarsenings of the canonical
henselian valuation. 
Recall that on a henselian valued field, any two henselian valuations with non-separably closed
residue field are comparable.

The \emph{canonical henselian valuation} $v_K$ on $K$
is defined as follows: If $K$ admits a henselian valuation 
with separably closed residue field, then $v_K$ is the (unique) coarsest such. In this case, any
henselian valuation with non-separably closed residue field is a proper coarsening
of $v_K$ and any henselian valuation with separably closed residue field
is a refinement of $v_K$.
If $K$ admits no henselian valuations with separably closed residue field, then $v_K$
is the (unique) finest henselian valuation on $K$ and any two henselian valuations
on $K$ are comparable.

\subsection{The canonical $p$-henselian valuation}
Let $K$ be a field and $p$ a prime. We define $K(p)$ to be the compositum of all Galois extensions
of $K$ of $p$-power degree. A valuation $v$ on $K$ is called
\emph{$p$-henselian} if $v$ extends uniquely to $K(p)$, furthermore, we say that $K$ is $p$-henselian
if it admits a non-trivial $p$-henselian valuation. 
Note that every henselian valuation is
$p$-henselian for all primes $p$ but, in general, not vice versa. 

Similarly to the henselian sitatuation, there is a canonical $p$-henselian valuation. Here,
one replaces the notion of `separably closed' by 
`admitting no Galois extensions
of degree $p$'.
Again, on a $p$-henselian field, any two $p$-henselian valuation whose residue fields admit Galois extensions
of degree $p$ are comparable.
The canonical henselian valuation $v_K^p$ on $K$
is defined as follows: If $K$ admits a $p$-henselian valuation 
with residue field not admitting Galois extensions of degree $p$, 
then $v_K^p$ is the (unique) coarsest such. In this case, any
$p$-henselian valuation with residue field admitting Galois extensions of degree $p$ 
is a proper coarsening
of $v_K^p$ and any $p$-henselian valuation whose residue field 
does not admit such extensions
is a refinement of $v_K^p$.
If there are no $p$-henselian valuations with residue field not admitting Galois extensions of degree
$p$ on $K$, then $v_K^p$
is the (unique) finest $p$-henselian valuation on $K$.
Whenever $K$ admits a non-trivial $p$-henselian valuation, 
$v_K^p$ is non-trivial and comparable to all $p$-henselian valuations on $K$.

Unlike the canonical henselian valuation, in most cases the canonical $p$-henselian 
valuation is
definable in $\mathcal{L}_{ring}$.
\begin{Thm}[Main Theorem in \cite{JK14}] 
Let $p$ be a prime. Consider the (elementary) class of fields
\label{MT}
\begin{align*}
\mathcal{K}_p:=\big\{ K \neq K(p) \;\mid\;\, & \zeta_p \in K \textrm{ in case }\mathrm{char}(K)\neq p \\
& \textrm{ and }
\sqrt{-1} \in K \textrm{ in case }p=2 \textrm{ and }\mathrm{char}(K)=0  \big\}.
\end{align*}
Then, the canonical $p$-henselian valuation is uniformly $\emptyset$-definable 
in $\mathcal{K}_p$, i.e.~there is a parameter-free $\mathcal{L}_{ring}$-formula $\phi_p(x)$ such that in any
$K \in \mathcal{K}_p$ we have $$\phi_p(K)=\mathcal{O}_{v_K^p}.$$ 
\end{Thm}

\section{Antiregular value groups} \label{sec:anti}
In this section, we use specific properties of the value group of the
canonical $p$-henselian valuation to define (henselian) coarsenings without
parameters. We first recall some work by Hong on defining valuations with
regular value groups which we make use of in some of our proofs. 
We then define a property of ordered abelian groups which 
we call antiregular and show a definability result for non-antiregular
value groups. Throughout the section, all quotients
of ordered abelian groups considered are assumed to be quotients 
by convex subgroups.
\begin{Def}
Let $\Gamma$ be an ordered abelian group and $p$ a prime. 
Then, $\Gamma$ is \emph{$p$-regular}
if all proper quotients of $\Gamma$ are $p$-divisible. Furthermore,
$\Gamma$ is \emph{regular} if it is $p$-regular for all primes $p$.
\end{Def}

Note that $p$-regularity is an elementary property of $\Gamma$:
$$\Gamma\textrm{ is } p\textrm{-regular }\Longleftrightarrow\;
\forall \gamma_0, \dotsc, \gamma_p\,(\gamma_0 < \dots < \gamma_p
\rightarrow \exists \delta\, (\gamma_0 \leq p\delta \leq \gamma_p))$$
Furthermore, an ordered abelian group is regular if and only if it is
elementarily equivalent to an archimedean ordered group. See \cite{Zak}
for more details on ($p$-)regular ordered abelian groups.
Hong proved 
the following
definability results about ($p$-)henselian valuations with ($p$-)regular
value groups.
\begin{Thm}[{\cite[Theorems 3 and 4]{Hon14}}]
Let $(K, v)$ be a valued field. \label{Hong}
\begin{enumerate}
\item Assume that $(K,v)$ is $p$-henselian and that we have 
$\zeta_p \in K$ in case
$\mathrm{char}(K)\neq p$. If $vK$ is $p$-regular and not $p$-divisible, 
then $v$ is definable.
\item If $(K,v)$ is henselian and $vK$ is regular
but not divisible, then $v$ is $\emptyset$-definable. 
\end{enumerate}
\end{Thm}

We can use this theorem to give an example of a field admitting both
definable and non-definable non-trivial henselian valuations.
\begin{Ex} \label{two}
Consider the field $K=\mathbb{R}((\mathbb{Q}))((\mathbb{Z}))$ 
(for details on power series fields see
\cite[\S 4.2]{Efr}). This field
admits exactly two non-trivial henselian valuations: The power series 
valuation $v_1$ with residue field $\mathbb{R}((\mathbb{Q}))$ and value group
$\mathbb{Z}$ is henselian and has no non-trivial coarsenings as its value
group has (archimedean) rank $1$. 
Furthermore, as $\mathbb{R}$ is non-henselian, 
the power series valuation $u$ 
with value group $\mathbb{Q}$ and residue field $\mathbb{R}$ is the only 
non-trivial henselian valuation on the field $\mathbb{R}((\mathbb{Q}))$.
Thus, $v_1$ has exactly one henselian refinement $v_2$, namely the refinement
of $v_1$ by $u$, with value group 
$\mathbb{Z}\oplus \mathbb{Q}$ (ordered lexicographically) and residue field
$\mathbb{R}$.

As $v_1K$ is regular and non-divisible, $v_1$ is $\emptyset$-definable by 
Theorem \ref{Hong}. We claim that $v_2$ is not $\emptyset$-definable:
Note that we have 
$\mathbb{R}\equiv \mathbb{R}((\mathbb{Q}))$ in $\mathcal{L}_{ring}$
since $\mathbb{R}((\mathbb{Q}))$ is also real closed
(see \cite[Lemma 4.3.6 and Theorem 4.3.7]{EP05}). 
Furthermore, 
there is an elementary equivalence of lexicographically ordered sums 
$\mathbb{Z}\oplus \mathbb{Q}\oplus\mathbb{Q}\equiv\mathbb{Z} \oplus\mathbb{Q}$
in $\mathcal{L}_{oag} = \{+,<,0\}$
since finite lexicographic sums
preserve elementary equivalence (\cite[proof of Theorem 3.3]{Gir88}) and the 
$\mathcal{L}_{oag}$-theory 
of divisible ordered abelian groups is complete (\cite[Corollary 3.1.17]{Mar}).
The Ax-Kochen/Ersov Theorem (\cite[Theorem 4.6.4]{PD}) implies that 
$$(K,v_2) 
\equiv (\mathbb{R}((\mathbb{Q}))\underbrace{((\mathbb{Q}))((\mathbb{Z}))}_{w_1},
w_1)
\equiv (\mathbb{R}
\underbrace{((\mathbb{Q}))((\mathbb{Q}))((\mathbb{Z}))}_{w_2},w_2)
$$ 
holds.
Thus, $v_2$ cannot be $\emptyset$-definable: Any parameter-free
first-order definition of $v_2$ would have to define both $w_1$ and $w_2$
on the field $\mathbb{R}((\mathbb{Q}))((\mathbb{Q}))((\mathbb{Z}))$.

Moreover, $v_2$ is not even definable with parameters: 
By \cite[Theorem 4.4 and Remark 3 on p.~1147]{DF96}, on any field $K$ the only
possible definable henselian valuation with real closed residue field
is the coarsest such. As $v_1$ is a proper coarsening
of $v_2$ with real closed residue
field, $v_2$ is not definable.
\end{Ex}

We now define an antipodal property to $p$-regularity.
\begin{Def}
Let $\Gamma$ be an ordered abelian group and $p$ a prime. 
Then, $\Gamma$ is \emph{$p$-antiregular} 
if no non-trivial quotient of $\Gamma$ is
$p$-divisible and $\Gamma$ has no rank-$1$ quotient.
Furthermore, $\Gamma$ is \emph{antiregular} if it is
$p$-antiregular for all primes $p$.
\end{Def}
Here, an ordered abelian group has rank $1$ if its
archimedean rank is $1$. 
Again, $p$-antiregularity is an elementary property of $\Gamma$:
$$\Gamma\textrm{ is } p\textrm{-antiregular }\Longleftrightarrow\;
\forall \gamma\, \exists \delta\, \forall \varepsilon\,
(|\varepsilon| \leq p|\gamma| \rightarrow \delta + \varepsilon 
\notin p \Gamma)$$
with the standard notation $|\gamma| := \mathrm{max}\{\gamma, {-\gamma}\}$.

\begin{Ex} \emph{Antiregular ordered abelian groups:}
For $i \in \mathbb{Z}$, let $Z_i$ be a copy of $\mathbb{Z}$ as
an ordered abelian group.
Consider the lexicographically ordered sums 
$$\Gamma := \bigoplus_{i \in \mathbb{Z}}Z_i \textrm{ and }
\Delta := \bigoplus_{i \in \mathbb{Z},\; i \leq 0}Z_i. $$
Then both $\Gamma$ and $\Delta$ are antiregular, as all of their
non-trivial 
quotients are either isomorphic to $\Gamma$ or $\Delta$, so in particular no
quotient is $p$-divisible
for any prime $p$ nor of rank $1$. 
The element $(\dots,0,0,0,1) \in \Delta$ is a minimal positive element, and
$\Gamma$ has no minimal positive element. Thus, we have 
$$\Gamma \not\equiv \Delta$$
as ordered abelian groups.
Note that any ordered abelian group which has an antiregular quotient is
again antiregular, so there are many examples of elementary classes of
antiregular ordered abelian groups.
\end{Ex}

\begin{Q} Is there a similar (first-order) 
classification for antiregular ordered
abelian groups as there is for regular ones? 
\end{Q}

We now collect some useful facts about antiregular ordered abelian groups.
\begin{Lem} \label{infty}
Let $\Gamma \neq \{0\}$ be an ordered abelian group. 
\begin{enumerate}
\item If $\Gamma$ is
$p$-antiregular, then we have $[\Gamma:p\Gamma]=\infty$.
\item If $\Gamma \leq \Gamma'$ and the index $[\Gamma':\Gamma]$ is finite,
then $\Gamma$ is $p$-antiregular if and only if $\Gamma'$ is 
$p$-antiregular.
\end{enumerate}
\end{Lem}
\begin{proof}
\begin{enumerate}
\item If $[\Gamma:p\Gamma]=n$, let $\{x_1,\dotsc x_n\}$ be 
a system of representatives for $\Gamma/p\Gamma$. Consider the convex
subgroup $\Delta \leq \Gamma$ generated by $\{x_1,\dotsc, x_n\}$. 
Then, the quotient $\Gamma/\Delta$ is $p$-divisible. If the quotient is 
trivial, then
$\Gamma$ has finite archimedean rank (and hence a rank-$1$ quotient), 
otherwise $\Gamma$ has a non-trivial $p$-divisible
quotient. In either case, $\Gamma$ is not $p$-antiregular.
\item Consider ordered abelian groups $\Gamma \leq \Gamma'$ with
$[\Gamma':\Gamma]$ finite. Then, there is a one-to-one
correspondence between
convex subgroups $\Delta'$ of $\Gamma'$ and convex subgroups $\Delta$ of
$\Gamma$ with $\Delta'\cap \Gamma = \Delta$ and furthermore
$[\Delta':\Delta]$ finite. 
In particular, $\Gamma'/\Delta'$ is $p$-divisible if and only if 
$\Gamma/\Delta$ is
$p$-divisible and $\Gamma'/\Delta'$ has rank $1$ if and only if 
$\Gamma/\Delta$ has rank $1$.
Thus, $\Gamma'$ is $p$-antiregular if and only if $\Gamma$ is.
\end{enumerate}
\end{proof}

The next lemma gives the means to define a coarsening of a $\emptyset$-definable
valuation with non-antiregular value group without parameters. 
\begin{Lem} \label{Delta}
Let $\Gamma$ be an ordered abelian group and $p$ prime. 
Define $$D :=\Set{ \Delta \leq \Gamma | \Delta \textrm{ convex and }
\Gamma/\Delta \textrm{ is }p\textrm{-divisible}}.$$
Then, we have:
\begin{enumerate}
\item $\Delta_0 := \bigcap\limits_{\Delta \in D} \Delta$ is a convex subgroup
of $\Gamma$ such that $\Gamma/\Delta_0$ is $p$-regular.
\item If $\Gamma \neq p\Gamma$ holds and every $p$-regular quotient
of $\Gamma$ is $p$-divisible then
$\Delta_0$ is $\emptyset$-definable: For any $\gamma \in \Gamma$, we have
$$\gamma \in \Delta_0 \Longleftrightarrow
\exists \varepsilon \forall \alpha\, (|\alpha|< |\gamma| \rightarrow
\varepsilon-\alpha \notin p\Gamma).$$
\end{enumerate}
\end{Lem}
\begin{proof}
\begin{enumerate}
\item Since all convex subgroups of $\Gamma$ are linearly ordered by inclusion,
it is clear that $\Delta_0$ is a convex subgroup of $\Gamma$. Every non-trivial
convex subgroup of $\Gamma/\Delta_0$ is of the shape $\Delta/\Delta_0$ for
some $\Delta \in D$. Hence, 
$$(\Gamma/\Delta_0) \big/ (\Delta/\Delta_0) \cong \Gamma/\Delta$$
is $p$-divisible and so $\Gamma/\Delta_0$ is $p$-regular.
\item As all $p$-regular quotients of
$\Gamma$ are $p$-divisible by assumption, 
$\Gamma/\Delta_0$ is $p$-divisible. 
Assume $\gamma \in \Delta_0$. Let $\langle \gamma \rangle$ be
the convex hull of the subgroup generated by $\gamma$ in $\Delta_0$.

We claim that $\Delta_0 / \langle \gamma \rangle$ is not $p$-divisible.
Assume for a contradiction that $\Delta_0 / \langle \gamma \rangle$
was $p$-divisible. Then, as $\Gamma/\Delta_0$ is $p$-divisible,
we also get that $\Gamma/\langle \gamma \rangle$ is $p$-divisible.
This implies $\langle \gamma \rangle \in D$ and hence 
$\Delta_0=\langle \gamma \rangle$. In particular, as $\Delta_0$ is not 
$p$-divisible, we get $p \nmid \gamma$. 
Consider the maximal convex subgroup
$B_\gamma$ of $\Delta_0$ such that $\gamma \notin B_\gamma$, i.e.
$$B_\gamma:=\Set{ \delta \in \Delta_0 | \forall n \in \mathbb{Z}:\,
|n\cdot \delta|<\gamma }.$$
Now, $\Delta_0/B_\gamma$ is a non-$p$-divisible rank-$1$ quotient
of $\Delta_0$ and thus of $\Gamma$. Hence, $\Gamma$ has a non-$p$-divisible
$p$-regular quotient, contradicting our assumption on $\Gamma$ that no such
exists. This proves the claim.

By the claim, we can choose some $\varepsilon \in \Delta_0 \setminus
\langle \gamma \rangle$  such that
$$\varepsilon + \langle \gamma \rangle \notin p \big(\Delta_0/
\langle \gamma \rangle \big)$$
holds. Hence, for any $\alpha \in \langle \gamma \rangle$, we have
$$\varepsilon-\alpha  \notin p\Delta_0 = p\Gamma \cap \Delta_0.$$
Thus, we have for all $\gamma \in \Delta_0$ 
$$\Gamma \models 
\exists \varepsilon \forall \alpha\, (|\alpha|< |\gamma| \rightarrow
\varepsilon-\alpha \notin p\Gamma).$$
Conversely, if $\gamma \notin \Delta_0$ holds then we have $\gamma \notin
\Delta$ for some $\Delta \in D$. As $\Gamma/\Delta$ is $p$-divisible, 
for every
$ \varepsilon \in \Gamma$ there is some $\alpha \in \Delta$ such that
$\varepsilon - \alpha \in p\Gamma$ holds. Thus, we have
for all $\gamma \in \Gamma \setminus \Delta_0$
$$\Gamma \models 
\forall \varepsilon \exists \alpha\, (|\alpha|< |\gamma| \wedge 
\varepsilon-\alpha \in p\Gamma).$$
\end{enumerate}
\end{proof}

\begin{Rem}
Let $\Gamma$ be an ordered abelian group and $p$ a prime
as in the assumptions of Lemma \ref{Delta}(2), i.e.,
assume that $\Gamma\neq p\Gamma$ holds
and that every $p$-regular quotient of $\Gamma$
is $p$-divisible. Define $\Delta_0$ as before. 
An alternative way to show that 
$\Delta_0$ is $\emptyset$-definable is to check that one has
$$\Delta_0=\bigcup_{\alpha \in \mathcal{S}_p} \Gamma_\alpha,$$
for $\mathcal{S}_p$ and $\Gamma_\alpha$ as defined in 
\cite[Definition 1.1]{Immi}.
\end{Rem}

We can now prove our first result on defining henselian valuations
without parameters.
\begin{Prop} \label{antireg}
Let $(K,v)$ be a henselian field and $p$ a prime. 
If the value group $vK$ 
is not $p$-divisible and not $p$-antiregular, then
some non-trivial (henselian) coarsening of $v$ is
$\emptyset$-definable on $K$.
\end{Prop}
\begin{proof} Assume that $vK$ is not $p$-divisible and not $p$-antiregular.
In case $\mathrm{char}(K)\neq p$, we may assume that $K$ contains a primitive 
$p$th root of unity $\zeta_p$: As $v$ is henselian, it extends uniquely to
a henselian valuation $w$ on
$F:=K(\zeta_p)$. By Lemma \ref{infty}, the value group $wF$ of the
prolongation is again non-$p$-divisible and not $p$-antiregular. As
$K(\zeta_p)$ is $\emptyset$-interpretable in $K$, any parameter-free
definition of a non-trivial coarsening of $w$ gives rise to a parameter-free
definition of a non-trivial coarsening of $v$.
In particular, the non-$p$-divisibility of $vK$ now implies $K \neq K(p)$.

If $vK$ admits a non-$p$-divisible rank-$1$ quotient, then
the corresponding coarsening is $\emptyset$-definable by Theorem \ref{Hong}.

Otherwise, $vK$ admits some non-trivial $p$-divisible quotient
by assumption.
If $vK$ admits a non-$p$-divisible $p$-regular quotient, then the 
corresponding coarsening is definable by Theorem \ref{Hong}, say via the
formula $\phi(x,t)$ for some parameter $t \in K$.
Note that $vK$ has at most one non-$p$-divisible $p$-regular quotient
and that no proper refinement of $v$ has $p$-regular value group. 
In particular, there is only one $p$-henselian valuation with 
non-$p$-divisible $p$-regular value group on $K$. 
By \cite[Theorem 1.5]{Koe95}, 
$p$-henselianity is an 
$\mathcal{L}_\mathrm{ring}$-elementary property of a valuation ring.
Thus,
the set
\begin{align*}
X = \big\{ \, t \in K \;|\; &
\mathcal{O}_{w_t}\coloneqq \phi(K, t) \textrm{ is a }
p\textrm{-henselian valuation ring }\\
&\textrm{with } w_tK\neq p\cdot w_tK \textrm{ and }
w_tK\textrm{ $p$-regular} \,\big\}
\end{align*}
is $\emptyset$-definable. Hence, the parameter-free formula
$$\psi(x) \equiv \exists t \in X \,(x \in \phi(K,t))$$
defines the unique $p$-henselian valuation with 
non-$p$-divisible $p$-regular value group on $K$ which is a non-trivial
coarsening of $v$.

Finally, assume that the value group $vK$ of $v$ only has 
$p$-divisible $p$-regular quotients; in particular, $vK$ is not
$p$-regular.
As $v$ is henselian, $v$ is comparable to the canonical $p$-henselian 
valuation $v_K^p$. In case $v_K^p$ is a coarsening of $v$, we have found
an $\emptyset$-definable coarsening of $v$.
Otherwise, the 
value group of the canonical
$p$-henselian valuation
$v_K^pK$ also admits only $p$-divisible $p$-regular quotients. Thus, Lemma
\ref{Delta} applies and $\Delta_0$ is $\emptyset$-definable in $v_K^pK$.
Now, the corresponding non-trivial 
$\emptyset$-definable coarsening $w$ of $v_K^p$ has $p$-divisible
value group and is hence also a coarsening of $v$. 

Note that any coarsening of a henselian valuation is again henselian. Thus, 
we have shown that if $v$ is henselian and 
$vK$ is non-$p$-divisible and not $p$-antiregular, then
some non-trivial, henselian coarsening of $v$ is $\emptyset$-definable.
\end{proof}

Next, we repeat the construction given in \cite{PZ78} of a field which is 
elementarily
equivalent in $\mathcal{L}_{ring}$ to a henselian field but which 
does not admit any non-trivial henselian valuation. Following
\cite{PZ78}, we define:
\begin{Def}
A field $K$ is called \emph{$t$-henselian} if there is some henselian field
$L$ with $L \equiv K$.
\end{Def}

Fields which are $t$-henselian but non-henselian play 
an important role in several of the examples in this paper:
Consider a field $K$ which is $t$-henselian but not henselian.
Clearly, no field elementarily equivalent to $K$ can admit a 
\emph{$\emptyset$-definable} non-trivial henselian valuation. 
However, for the field $K$ as discussed in the following example, 
any henselian field elementarily equivalent to $K$ admits a 
\emph{parameter-definable} henselian valuation (see Example \ref{PZ1}).
This follows from the fact 
that the canonical $2$-henselian valuation $v_K^2$
is $\emptyset$-definable and has an antiregular value group.
\begin{Ex} \emph{A $t$-henselian field which is not henselian:}
\label{PZ} Let 
$K_0:=\mathbb{Q}^{alg}$, and let $v_0$ be the trivial valuation on $K_0$.
For $n\geq 1$, one iteratively constructs valued fields $(K_n,v_n)$ with 
$v_nK_n = \mathbb{Z}$ and $K_nv_n = K_{n-1}$ and such that Hensel's Lemma
holds for polynomials of degree at most $n$ as follows:
Choose a minimal  algebraic extension $K_n$ of $K_{n-1}(X_{n-1})$ with
$$K_{n-1}(X_{n-1}) \subseteq K_n \subsetneq 
K_{n-1}((X_{n-1})),$$ 
such that Hensel's Lemma holds on $(K_n,v_n)$ for polynomials of degree at
most $n$, 
where $v_n$ is the restriction of the power series valuation on 
$K_{n-1}((X_{n-1}))$
to $K_n$. One can of course choose $K_1=K_0(X_0)$ with $v_1$ the $X_0$-adic
valuation, as Hensel's Lemma holds trivially for all polynomials of degree $1$.
Note that we get a place $p_n: K_n \to K_{n-1} \cup \{\infty\}$
with is $p$-henselian for all primes $p \leq n$.

The field $K$ is then taken as the inverse limit of 
$$(K_n \cup \infty, p_n) \textrm{ with projections }s_n:K \cup\{\infty\} \to
K_{n-1} \cup \{\infty\}.$$
It follows from the arguments given in \cite[p.~338]{PZ78} that $K$ admits
no non-trivial henselian valuation.

The canonical $2$-henselian valuation $v_K^2$ on $K$ now corresponds to the 
place $$s_2: K \to K_1 \cup \{\infty\}$$ 
as $p_n$ is $2$-henselian if and only if $n\geq 2$.
As usual, the quotients of $v_K^2K$ correspond to the value groups
of coarsenings of $v_K^2$.
Since the coarsenings of $v_K^2$ correspond to the places $s_n$ for $n \geq 2$
and none of them has a $p$-divisible value group for any prime $p$ or has 
a value group of rank-$1$, we conclude that 
the group $v_K^2K$ is antiregular.
\end{Ex}

\section{Defining coarsenings of valuations using subgroups} 
In this section, \label{sec:AEJ} 
we discuss a class of parameter-definable convex subgroups
of an ordered abelian group. The motivation for this comes from \cite{AEJ87}.
We then apply our construction to show that a field admitting
a henselian valuation with non-divisible value group admits
a non-trivial parameter-definable henselian valuation.
\begin{Lem} Let $\Gamma$ be an ordered abelian group and $p$ a prime. \label{Dg}
Take any $\gamma \in \Gamma$
with $\gamma>0$ and define 
$$\Delta_\gamma \coloneqq \Set{ \delta \in \Gamma | [0,p\cdot|\delta|] 
\subseteq [0,p\cdot \gamma]+p\Gamma}$$
where $|\delta|=\mathrm{max}\{\delta,-\delta\}$. Then
$\Delta_\gamma$ is a convex $\{\gamma\}$-definable subgroup of $\Gamma$ with 
$\gamma\in \Delta_\gamma$. Furthermore, no non-trivial convex subgroup of
$\Gamma/\Delta_\gamma$ is $p$-divisible.
\end{Lem}
\begin{proof} By definition, $\Delta_\gamma$ is a 
$\{\gamma\}$-definable convex subset of $\Gamma$ containing $\gamma$ with 
$\Delta_\gamma = -\Delta_\gamma$.
We now show that $\Delta_\gamma$ is a subgroup of $\Gamma$.
As $\Delta_\gamma$ is convex, it suffices to show that for all 
 $\delta \in \Delta_\gamma$ we have $\delta+\delta \in \Delta_\gamma$.
Since we have $\Delta_\gamma = -\Delta_\gamma$, it suffices to consider
the case $\delta>0$.
Take any $\delta \in \Delta_\gamma$ with $\delta>0$ and $\beta \in 
\Gamma$ with 
$$0 \leq |\beta| \leq p \cdot (\delta+ \delta).$$
In case we have $|\beta| \leq p \cdot \delta$ we immediately get 
$|\beta| \in [0,p \cdot \gamma]+p\Gamma$. Otherwise, we have
$p\cdot \delta < |\beta| \leq p \cdot (\delta+ \delta)$,
so we get $|\beta|-p\delta \leq p\delta$. This implies again
$|\beta| \in [0,p \cdot \gamma]+p\Gamma$.
Overall, we get $[0,p\cdot (\delta+\delta)] \subseteq 
[0,p \cdot \gamma]+p\Gamma$, i.e., $\delta+\delta \in \Delta_\gamma$
as required.

Let $\tilde{\Delta} \leq \Gamma$ be a convex subgroup with 
$\Delta_\gamma \subseteq \tilde{\Delta}$. 
If $\tilde{\Delta}/\Delta_\gamma$ 
is $p$-divisible, then for any $\tilde{\delta} \in \tilde{\Delta}$ there is
some $\delta \in \Delta_\gamma$ with $\tilde{\delta}-\delta\in p\Gamma$. 
Fix some $\tilde{\delta} \in \tilde{\Delta}$ and take any
$\tilde{\beta} \in [0,p\cdot|\tilde{\delta}|]\subseteq \tilde{\Delta}$.
Then, there is some $\beta \in \Delta_\gamma$ with
$$\tilde{\beta}\in \beta + 
p\Gamma \subseteq [0,p\cdot\gamma]+p\Gamma.$$
Thus, we get $\tilde{\delta} \in \Delta_\gamma$ and hence $\tilde{\Delta}
=\Delta_\gamma$. As any convex subgroup of $\Gamma/\Delta_\gamma$ corresponds
to a subgroup $\tilde{\Delta}\leq \Gamma$ as above, we conclude that 
$\Gamma/\Delta_\gamma$ has no non-trivial
$p$-divisible convex subgroup. 
\end{proof}

If $\Gamma$ is the value group of a definable valuation $v$ on a field $K$, the
construction in the Lemma gives rise to a definable coarsening of $v$. 
As discussed
in the next remark, this is a special case of a construction
introduced by Arason, Elman and Jacob (see \cite{AEJ87}).
\begin{Rem}
Let $(K,v)$ be a valued field and $t \in \mathfrak{m}_v$. Consider
the set
$$T_t \coloneqq \Set{ x \in K^\times | \exists z:\, v(t^{-p})\leq v(xz^p)\leq 
v(t^p)}.$$
It is straightforward to check that if $v$ is $\emptyset$-definable, 
$T_t$ is a $t$-definable
subgroup of $K^\times$. 
In \cite{AEJ87}, the authors introduce a method 
how to obtain definable valuation
rings from certain definable subgroups of $K^\times$ and discuss conditions
under which this valuation ring is non-trivial.
Using the notation and machinery from 
\cite{AEJ87} (in particular Theorem 2.10 and Lemma 3.1), 
one can show that there is a valuation ring $\mathcal{O}(T_t,T_t) \subseteq K$
which is trivial if and only of $T_t = K^\times$.

Now, let $\Delta_{v(t)}$ be the convex subgroup of $vK$ as defined in Lemma 
\ref{Dg}. The valuation ring $\mathcal{O}(T_t,T_t)$ is exactly the coarsening
of $v$ which is obtained by quotienting $vK$ by 
the convex subgroup $\Delta_{v(t)}$. This valuation can also be described as
the finest coarsening $w$ of $v$ such that we have $t \in \mathcal{O}_w^\times$
and such that no non-trivial convex subgroup of $wK$ is $p$-divisible.
\end{Rem}

Proposition \ref{antireg} 
and Lemma \ref{Dg} are the two main ingredients needed 
to show that on any field admitting a henselian valuation with non-divisible
value group there is a non-trivial definable henselian valuation: 
\begin{Prop} \label{nondiv}
Assume that some henselian valuation $v$ on $K$ has a 
non-divisible value group. 
Then, some non-trivial (henselian) coarsening of $v$ is definable on $K$
(using at most $1$ parameter).
\end{Prop}
\begin{proof}
Let $(K,v)$ be henselian such that $vK$ is not $p$-divisible for some 
prime $p$.
If $vK$ is not $p$-antiregular, then it admits a $\emptyset$-definable 
non-trivial coarsening by
Proposition \ref{antireg}. Thus, we may assume that $vK$ is $p$-antiregular
which means that it has no non-trivial $p$-divisible quotient and 
no rank-$1$ quotient. 

Consider $F:=K(\zeta_p)$ in
case $\mathrm{char}(K)\neq p$: By Lemma \ref{infty},
the unique prolongation $w$ of $v$ to 
$K(\zeta_p)$ will again have non-$p$-divisible and $p$-antiregular value group. 
If we define 
a coarsening of $w$ with parameters from $K$ on $F$, 
its restriction to $K$ is 
also definable (with parameters from $K$).
If $\mathrm{char}(K)= p$, we set $F:=K$.

We now have $F \neq F(p)$ and -- as $[F:K]$ is prime to $p$ -- 
for any $t \in \mathfrak{m}_v$ with $p \nmid v(t)$ we get $p \nmid w(t)$.
By construction, $v_F^p$ is $\emptyset$-definable on $F$. 
Note that, by henselianity, $w$ is comparable to $v_F^p$. If $v_F^p$ is a
coarsening of $w$, we have found a non-trivial $\emptyset$-definable 
coarsening of $w$ (and thus of $v$). Hence, we may assume that $v_F^p$ refines
$w$. 

Choose any $t \in \mathfrak{m}_{v} \subseteq \mathfrak{m}_w$ 
with $p \nmid v(t)$. Then, we also have 
$p \nmid w(t)=: \gamma$. Define $\Gamma\coloneqq wF$ and consider
the convex subgroup 
$$\Delta_{\gamma}=\Set{\delta \in \Gamma | [0,p\cdot|\delta|] 
\subseteq [0,p\cdot \gamma]+p \Gamma } $$
of $\Gamma$ as in Lemma \ref{Dg}. We claim that 
$\Delta_{\gamma} \neq \Gamma$ holds.
Assume for a contradiction that we have $\Delta_\gamma=\Gamma$.
Let $\langle \gamma \rangle$ denote
the convex subgroup of $\Gamma$ generated by $\gamma$.
Then, we have for all $\delta \in \Delta_\gamma = \Gamma$ that
$$|\delta| \in [0,p\cdot\gamma]+p\Gamma \subseteq
\langle \gamma \rangle + p\Gamma$$
holds. Thus, $\Gamma/\langle \gamma \rangle$ is $p$-divisible,
and thus -- as $\Gamma$ is $p$-antiregular --  
trivial.
Now, the maximal convex subgroup of $\Gamma$ not containing $\gamma$, i.e.
$$B_\gamma \coloneqq 
\Set{ \delta \in \Gamma | \forall n \in \mathbb{Z}:\, 
|n \cdot \delta| < \gamma },$$
is a proper convex subgroup of $\Gamma$ such that $\Gamma/B_\gamma$ 
has rank $1$. This 
contradicts the $p$-antiregularity of $\Gamma$. Thus, we conclude
$\Gamma \neq \Delta_\gamma$.

Hence, the coarsening of $w$ which corresponds to quotienting $wF$ 
by $\Delta_\gamma$ is a non-trivial $\{t\}$-definable coarsening of $w$. Its
restriction to $K$ is a non-trivial $\{t\}$-definable coarsening of $v$.
\end{proof}

\section{Definitions without parameters} \label{sec:0}
We are now in a position to prove our main theorem 
on
the existence of a parameter-free definable henselian valuation 
on a henselian field
as stated in the introduction:
{
\renewcommand{\theTh}{A}
\begin{Thm}
Let $K \neq K^\mathrm{sep}$ be a henselian field. Then $K$ admits a
$\emptyset$-definable non-trivial henselian valuation unless
\begin{enumerate}
\item $Kv_K \neq Kv_K^\mathrm{sep}$ \emph{and}
\item $Kv_K$ is $t$-henselian \emph{and}
\item $v_KK$ is $p$-antiregular for all 
primes $p$ with $v_KK \neq p v_KK$ (e.g., if $v_KK$ is divisible). 
\end{enumerate}
\end{Thm}
\addtocounter{Th}{-1}
}

\begin{proof}
Let $K \neq K^\mathrm{sep}$ be a henselian field. In case
\begin{enumerate}[($\neg$1)]
\item $Kv_K = Kv_K^\mathrm{sep}$, then $K$ admits a $\emptyset$-definable
henselian valuation by \cite[Theorem 3.10]{JK14a}. 
\item $Kv_K$ is not $t$-henselian, then $K$ admits a
$\emptyset$-definable henselian valuation by \cite[Proposition 5.5]{FJ14}.
\item 
$v_KK$ is non-$p$-divisible and not $p$-antiregular for 
some $p$ 
then some non-trivial (hen\-se\-lian) coarsening of $v_K$ is
$\emptyset$-definable by Proposition \ref{antireg}.
\end{enumerate}
Hence, if $K$ does not satisfy one of the conditions (1)-(3) occuring 
in the theorem, then $K$
admits a non-trivial $\emptyset$-definable henselian valuation.
\end{proof}
Note that real closed fields are subsumed in the `unless' setting: If $K$
is henselian and real closed, then $Kv_K$ is real closed (and thus $t$-henselian
and not separably closed) and $v_KK$ is divisible.

We first draw some conclusions from Theorem \ref{1}.
\begin{Cor} \label{Cor1}
Let $K$ be a non-separably closed 
henselian field with $G_K$ small or $\mathrm{trdeg}(K)$ finite. 
Then $K$ admits a
$\emptyset$-definable henselian valuation 
unless $v_KK$ is divisible and $Kv_K$ is $t$-henselian but not separably closed.
\end{Cor}
\begin{proof}
Let $K$ be henselian and assume $K\neq K^\mathrm{sep}$. 
If $K$ admits no
$\emptyset$-definable henselian valuation then, by Theorem \ref{1},
$Kv_K$ is $t$-henselian but not separably closed. 
If $v_KK$ is not divisible, Theorem \ref{1} implies that 
$v_KK$ is $p$-antiregular and not $p$-divisible for at least one prime $p$.

Let $K$ be a field of finite transcendence degree or such that $G_K$ is small. 
Then, the index $[v_KK:pv_KK]$ is finite for any prime $p$.
Hence, Lemma \ref{infty} implies that $v_KK$ is not $p$-antiregular.
\end{proof}

\begin{Cor} \label{cor:fin}
Let $K$ be a henselian field, neither separably closed nor real
closed, and assume 
$\mathrm{trdeg}(K)$ finite. 
Then $K$ admits a
$\emptyset$-definable henselian valuation.
\end{Cor}
\begin{proof}
Assume $\mathrm{trdeg}(K)$ finite and $K \neq K^\mathrm{sep}$. 
By Corollary \ref{Cor1}, $K$ admits a $\emptyset$-definable henselian 
valuation unless $Kv_K$ is $t$-henselian but not henselian. However,
\cite[Theorem 3.4.2]{EP05} 
implies that $\mathrm{trdeg}(Kv_K)$ is also finite. By
\cite[Lemma 3.5]{Koe04}, 
every $t$-henselian field of finite transcendence degree is
henselian. Thus, $Kv_K$ cannot be $t$-henselian but not henselian.
\end{proof}

\begin{Cor} \label{cor:small}
Let $K$ be a henselian field with $G_K$ small and $\mathrm{char}(Kv_K)=0$. 
Then $K$ admits no
$\emptyset$-definable henselian valuation iff $K\equiv Kv_K$.
\end{Cor}
\begin{proof}
Let $K$ be a henselian field with $G_K$ small, $\mathrm{char}(Kv_K)=0$,
which does not admit a
$\emptyset$-definable henselian valuation. 
Corollary \ref{Cor1} implies that 
$v_KK$ is divisible and $Kv_K$ is $t$-henselian. 
By \cite[Lemma 3.3]{PZ78}, there 
is some henselian $L \succ Kv_K$. Note that $G_{Kv_K}$, and hence $G_L$, 
is also small. Using Corollary \ref{Cor1} once more, we get that 
$Lv_L$ is $t$-henselian and
$v_LL$ is divisible. Since the restriction of $v_L$ to $Kv_K$ is trivial,
 we have $\mathrm{char}(Lv_L)=0$. Using
the Ax-Kochen/Ersov Theorem (\cite[Theorem 4.6.4]{PD}) several times, 
we conclude
$$Kv_K \equiv L \equiv Lv_L((\mathbb{Q}))\equiv Lv_L((\mathbb{Q}))((\mathbb{Q})) \equiv
L((\mathbb{Q})) \equiv Kv_K((\mathbb{Q})) \equiv K.$$
On the other hand, if $K \equiv Kv_K$, we have that $K$ is either 
separably
closed (and hence admits no non-trivial 
$\emptyset$-definable henselian valuation) 
or -- by the definition of $v_K$ -- 
that $Kv_K$ is $t$-henselian but not henselian. In the latter case,
$K$ cannot admit a $\emptyset$-definable non-trivial henselian valuation as
otherwise $Kv_K$ would be henselian.
\end{proof}

Note that by \cite[Construction 6.5 and Proposition 6.7]{FJ14}, 
there are fields with small absolute Galois group 
which are $t$-henselian but not henselian. Hence, there are 
henselian fields with small absolute Galois group which admit no 
non-trivial $\emptyset$-definable henselian valuation. Furthermore, 
Example \ref{FJ} shows that there are henselian fields with small
absolute Galois group not admitting any non-trivial definable henselian 
valuation.

We now give an example illustrating that, in general, 
Theorem \ref{1} does not give rise to
a full classification which henselian fields admit $\emptyset$-definable
henselian valuations:
\begin{Ex} \emph{A field admitting an $\emptyset$-definable henselian
valuation satisfying conditions $(1)$-$(3)$ in Theorem \ref{1}:}
\label{PZ1} \label{ex:anti}
Let $K$ be the field as constructed in Example \ref{PZ}, so $K$
is elementarily equivalent to a henselian field but not henselian
and $v_K^2$ is $\emptyset$-definable and 
$p$-antiregular value group for all $p$. 

Consider the canonical henselian valuation $v_L$
on $L=K((\mathbb{Q}))$. Note that $v_L$ is the power series valuation on $L$,
thus $v_LL=\mathbb{Q}$ is divisible and $Lv_L=K$ 
is $t$-henselian but not separably 
closed. In particular, it does not follow from Theorem \ref{1} that
$L$ admits a $\emptyset$-definable henselian valuation.

We claim that $v_L$ is $\emptyset$-definable. Fix any prime $p$.
As $K$ is $2$-henselian, 
$v_L^2$ refines $v_L$. Thus, $v_L^2L$ has a $p$-divisible quotient
(namely $\mathbb{Q}$) and is therefore
not $p$-antiregular.
Furthermore, $v_L^2$ is the composition of $v_L$ and 
$v_K^2$, so -- as $v_K^2K$ is $p$-antiregular --
$\mathbb{Q}$ is the only $p$-regular quotient of $v_L^2L$.
Hence, $v_L^2L$ has no non-$p$-divisible $p$-regular quotient
and so some non-trivial
convex subgroup with $p$-divisible quotient is $\emptyset$-definable in
$v_L^2L$ by Lemma \ref{Delta}.
However, $Lv_L$ is the only such quotient and $v_L^2$ is
$\emptyset$-definable by Theorem \ref{MT}. Thus, 
$v_L$ is $\emptyset$-definable. 
\end{Ex}

The arguments given in the Example above can in general be used to prove the 
following partial converse to Theorem \ref{1}:
\begin{Prop} \label{1.1}
Let $K$ be a henselian field with $\mathrm{char}(Kv_K)=0$. If
\begin{enumerate}
\item $Kv_K \prec L$ for some henselian $L$ with $v_LL$ non-divisible 
\emph{and}
\item $v_KK$ is divisible 
\end{enumerate}
then $K$ admits a $\emptyset$-definable non-trivial henselian valuation.
\end{Prop}
\begin{proof}
Let $K$ be a henselian field such that $Kv_K \prec L$ for some henselian $L$ 
with $v_LL$ non-divisible. Fix a prime $p$ with 
$v_LL \neq pv_LL$. 
Then, in particular, $L$ is not separably closed
and hence neither are $Kv_K$ nor $K$.
Since $Kv_K$ is $t$-henselian but not henselian, $L$ admits no 
$\emptyset$-definable non-trivial henselian valuation. Thus,
by Theorem \ref{1}, $v_LL$ is $p$-antiregular.

Consider the field $M := L((v_KK))$ with the power series valuation $w$. 
By the Ax-Kochen/Ersov Theorem \cite[Theorem 4.6.4]{PD}, 
we have $$(K,v_K) \equiv (M,w).$$
Note that $v_MM \equiv v_KK \oplus v_LL$ holds 
(with the sum ordered  
lexicographically). Therefore, 
$v_MM$ is not $p$-divisible and not $p$-antiregular.
Hence, $M$ admits a $\emptyset$-definable non-trivial henselian valuation
by Theorem \ref{1}. Thus, $K$ also admits a $\emptyset$-definable 
non-trivial henselian valuation.
\end{proof}

It would be very interesting to have a complete
classification for the existence of non-trivial
$\emptyset$-definable henselian valuation. A necessary condition is that
any elementarily equivalent field also admits a non-trivial henselian
valuation. We now ask whether this condition is also sufficient:
\begin{Q}
Let $K$ be a henselian field such that any $L \equiv K$ is henselian. 
Does $K$ admit a non-trivial $\emptyset$-definable henselian valuation? 
\end{Q}
It follows immediately from the Corollaries \ref{cor:fin} and 
\ref{cor:small} 
that if $K$ is a
henselian field of finite transcendence degree or which has a small absolute
Galois group such that additionally 
$(\mathrm{char}(K),\mathrm{char}(Kv_K))=(0,0)$ holds, 
the answer to this question is 
positive.

\section{Definitions with parameters} \label{sec:def} \label{sec:1}
{
\renewcommand{\theTh}{B}
\begin{Thm} 
Let $K \neq K^\mathrm{sep}$ be a henselian field. Then $K$ admits a
definable non-trivial henselian valuation (using at most $1$ parameter) 
unless
\begin{enumerate}
\item $Kv_K \neq Kv_K^\mathrm{sep}$ \emph{and}
\item $Kv_K \prec L$ for some henselian $L$ with $v_LL$
divisible \emph{and}
\item $v_KK$ is divisible. 
\end{enumerate}
\end{Thm}
\addtocounter{Th}{-1}
}
\begin{proof} 
Let $K \neq K^\mathrm{sep}$ be a henselian field.
By Theorem \ref{1}, we get an $\emptyset$-definable non-trivial 
henselian valuation on $K$
unless we have $Kv_K \neq Kv_K^\mathrm{sep}$, $Kv_K$ $t$-henselian, $v_KK$
is $p$-antiregular for all primes $p$ with
$v_KK \neq p\cdot v_KK$.

If $Kv_K$ is $t$-henselian, then there is some henselian $L \succ Kv_K$
by \cite[Lemma 3.3]{PZ78}. In case there is some such $L$ with $Lv_L$ non-divisible,
then $K$ again admits a non-trivial $\emptyset$-definable henselian valuation
by Proposition \ref{1.1}.

If $v_KK$ is not $p$-divisible for some prime $p$, 
some non-trivial (henselian) coarsening of $v_K$ is definable 
using at most $1$ parameter 
by Proposition \ref{nondiv}.

Thus, if one of the conditions $(1)$-$(3)$ fails for $K$, then $K$ admits
a definable non-trivial henselian valuation (using at most $1$ parameter).
\end{proof}

In equicharacteristic $0$, 
we can show a full converse to Theorem \ref{2}:
\begin{Cor} \label{0,0}
Let $K \neq K^\mathrm{sep}$ be a henselian field 
with $\mathrm{char}(Kv_K)=0$. Then $K$ admits a
definable non-trivial henselian valuation 
\emph{if and only if not}
\begin{enumerate}
\item $Kv_K \neq Kv_K^\mathrm{sep}$ \emph{and}
\item $Kv_K\prec L$ for some henselian $L$ with $v_LL$
divisible \emph{and}
\item $v_KK$ is divisible. 
\end{enumerate}
\end{Cor}
\begin{proof}
Let $K \neq K^\mathrm{sep}$ be a henselian field 
with $\mathrm{char}(Kv_K)=0$ such that we have 
\begin{enumerate}
\item $Kv_K \neq Kv_K^\mathrm{sep}$ \emph{and}
\item $Kv_K\prec L$ for some henselian $L$ with $v_LL$
divisible \emph{and}
\item $v_KK$ is divisible. 
\end{enumerate}
We need to show that $K$ admits no definable non-trivial henselian valuation.
The key argument of the proof is relative quantifier 
elimination in the Denef-Pas-language, however, we first need to do some work
to set the situation up.

Since we have $Kv_K \neq Kv_K^{sep}$, 
any henselian valuation is a coarsening of $v_K$. Take $L \succ Kv_K$
with $v_LL$ divisible. Note that as the extension $Kv_K \subset L$ is 
regular, the restriction of $v_L$ to $Kv_K$ is henselian and hence trivial.
Thus, we also get $\mathrm{char}(Lv_L)=0$.

We claim that we have $v_FF$ divisible for \emph{all} $F \succ Kv_K$.
Note that as $Kv_K$
is $t$-henselian but not henselian, no field elementarily equivalent to $Kv_K$
can admit a non-trivial $\emptyset$-definable henselian valuation. 
Furthermore, the Ax-Kochen/Ersov Theorem (\cite[Theorem 4.6.4]{PD}) implies 
$$Kv_K \equiv L \equiv Lv_L((\mathbb{Q})) \equiv Lv_L((\mathbb{Q}))((\mathbb{Q})
\equiv L((\mathbb{Q})) \equiv Kv_K((\mathbb{Q})).$$
Now, if there was some $F \succ Kv_K$ with $v_FF$ non-divisible, then --
by Proposition \ref{1.1} --
$Kv_K((\mathbb{Q}))$
would admit a $\emptyset$-definable
non-trivial henselian valuation, a contradiction. This proves the claim.

Consider an $\aleph_0$-saturated elementary extension $(M,v)$ of $(K,v_K)$
in $\mathcal{L}_{val}=\mathcal{L}_{ring}\cup \{\mathcal{O}\}$, where 
$\mathcal{O}$ is a unary predicate which is interpreted as the valuation ring.
Then, $vM$ is a divisible ordered abelian group and 
$F:=Mv$ is an $\aleph_0$-saturated 
elementary extension of $Kv_K$ in $\mathcal{L}_{ring}$ 
and thus henselian (\cite[Lemma 3.3]{PZ78}). In particular, $v_F$ is non-trivial
and hence $v_M$ is a proper refinement of $v$, namely the composition of $v$ 
and $v_F$. By the claim, we get that
$v_FF$ is divisible (and thus also $v_MM$). Note that the restriction of $v_M$ 
to $Kv_K$ is trivial, thus we get $\mathrm{char}(Mv_M)=0$.

We want to consider $(M,v_M)$ as a structure in the 
Denef-Pas-language $\mathcal{L}_{DP}$ 
which is an extension of $\mathcal{L}_{ring}$ (see \cite{Pas} for details).
A valued field $(N,w)$ can be made into an 
$\mathcal{L}_{DP}$-structure if and only if
there is a multiplicative map $\mathrm{ac}:N\to Nw$ with $\mathrm{ac}(0)=0$
and which coincides on $\mathcal{O}_w^\times$ with the residue map.
If there is no such map for $(M,v_M)$, there is an
$\mathcal{L}_{val}$-elementary
extension $(M,v_M) \prec (N,w)$ 
such that 
$(N,w)$ can be considered as an $\mathcal{L}_{DP}$-structure since
$(M,v_M)$ is henselian of equicharacteristic $0$. In particular,
we have $\mathrm{char}(Nw)=0$ and $wN$ divisible.

Assume for a contradiction that $K$ admits a definable non-trivial
henselian valuation, i.e., that some non-trivial coarsening of $v_K$
is definable. Then, via the elementary embedding
$$(K,v_K) \prec (M,v),$$ some non-trivial
coarsening of $v$ is $\mathcal{L}_{ring}$-definable on $M$ 
(using the same $\mathcal{L}_{ring}$-formula and the same parameters from 
$K\subseteq M$). 
Thus, some proper coarsening of
$v_M$ is $\mathcal{L}_{ring}$-definable in the henselian valued field 
$(M,v_M)$. Furthermore,
via the elementary embedding
$$(M,v_M) \prec (N,w),$$
some proper coarsening of
$w$ is $\mathcal{L}_{ring}$-definable in $N$.

In particular, this induces a definition of 
a proper, non-trivial convex subgroup of the divisible ordered abelian group 
$wN$.
 
Note that we have $\mathrm{char}(Nw)=0$. 
By the relative quantifier elimination
result in $\mathcal{L}_{DP}$
the following holds in a henselian
valued field $(N,w)$ of equicharacteristic $0$ (see \cite[Theorem 4.1]{Pas}): 
Any 
$\mathcal{L}_{Pas}$-definable subset of
$wN$ (using parameters from $N$) 
is already definable in the ordered abelian group $wN$ (using parameters from
$wN$). 
However, in a 
divisible ordered abelian group (like $wN$), there can be 
no proper, non-trivial convex definable subgroups. Hence, no non-trivial
proper coarsening of $w$ is definable on $N$ and thus there can be no
non-trivial definable henselian valuation on $K$. 
\end{proof}

\begin{Ex} \emph{A henselian field which does not admit
any non-trivial definable henselian valuation:} 
Refining the construction by Prestel and Ziegler as \label{FJ}
repeated in Example \ref{PZ}, one can construct a $t$-henselian
non-henselian field $k$ of characteristic $0$
with $k \neq k^{sep}$ and $G_k$ small (see \cite[Construction 6.5 and Proposition 6.7]{FJ14}).
By \cite[Proposition 5.8]{FJ14}, $v_LL$ is divisible 
for any henselian $L\succ k$. Consider the field $K:=k((\mathbb{Q}))$.
By Corollary \ref{0,0}, $K$ does not admit a non-trivial definable
henselian valuation.
\end{Ex}

\begin{Ex} \emph{A field $L$ admitting a non-trivial 
definable henselian valuation such that there is some 
non-henselian $K \equiv L$:} 
Consider the field $K$ as constructed in Example \ref{PZ}, so $K$
is $t$-henselian but not henselian,
$v_K^2$ is $\emptyset$-definable and 
has an antiregular value group. 

By \cite[Lemma 3.3]{PZ78}, there is some elementary extension $L \succ K$
such that $L$ is henselian. We claim that the canonical
 henselian valuation $v_L$
on $L$ has a non-divisible value group.
By Theorem \ref{MT}, $v_L^2$ and $v_K^2$ are defined by the same parameter-free
formula. As antiregularity is an elementary property of an
ordered abelian group, $v_L^2L$ is also antiregular. 
Since $v_LL$ is a quotient of $v_L^2L$, it cannot be $p$-divisible for any prime
$p$.

Thus, by Theorem \ref{2}, $L$ admits a non-trivial definable henselian
valuation. Since we have $L \equiv K$ and $K$ is $t$-henselian but not 
henselian, $L$ does not admit any non-trivial $\emptyset$-definable
henselian valuation.
\end{Ex}

\section{Tamely branching ($p$-)henselian valuations}
In this section, we study ($p$-)henselian valuations tamely branching at $p$. 
In the first part, 
we show that every field which admits a $p$-henselian 
valuation
tamely branching at $p$ admits a $\emptyset$-definable such and draw
some Galois-theoretic conclusions.
In the second part, we show that every field which admits 
a henselian valuation
tamely branching at $p$ admits a definable such, however, in general, 
parameters are required for the definition. We conclude that admitting a tamely
branching henselian valuation is not an elementary property in 
$\mathcal{L}_\textrm{ring}$, which has again 
some Galois-theoretic consequences. 

First, we recall the definition of tamely branching valuations. 
\label{sec:tame}
\begin{Def} 
Let $(K,v)$ be a valued field and $p$ a prime. 
We call $v$ \emph{tamely branching} at $p$ if
\begin{enumerate}
\item $\mathrm{char}(Kv) \neq p$ and
\item $vK$ is not $p$-divisible and 
\item if $[vK:pvK] = p$, then $Kv$ has a finite separable field extension
of degree divisible by $p^2$.
\end{enumerate}
\end{Def}

\subsection{Defining tamely branching $p$-henselian valuations}
We first consider the problem of defining $p$-henselian valuations
tamely branching at $p$. 
The existence of these valuations are encoded in the maximal pro-$p$ quotient
of the absolute Galois group of a field, as described
by following
\begin{Thm}[Engler, Koenigsmann and Nogueira; 
Theorem 2.15 in \cite{Koe03}] \label{EKN}
Let $p$ be a prime, $K$ a field containing a primitive $p$th root
of unity (in particular $\mathrm{char}(K)\neq p$) and assume
$G_K(p) \not\cong \mathbb{Z}_p$ and, if $p=2$, also
$G_K(p) \not\cong \mathbb{Z}/2\mathbb{Z}$ or $\mathbb{Z}_2 \ltimes
\mathbb{Z}/2\mathbb{Z}$. Then $K$ admits a $p$-henselian valuation tamely
branching at $p$ iff $G_K(p)$ has a non-trivial normal abelian subgroup.
\end{Thm}

We now turn to the definability of such valuations:
\begin{Prop} \label{ptame}
Let $K$ be a field and $p$ a prime such that $\mathrm{char}(K) \neq p$
holds.
Assume that 
we have $\zeta_p \in K$, and furthermore $\sqrt{-1} \in K$ in case $p=2$ and 
$\mathrm{char}(K)=0$.
If $K$
admits a $p$-henselian valuation tamely branching at $p$
then $K$ admits a $\emptyset$-definable such valuation.
\end{Prop}
\begin{proof}
Let $v$ be a $p$-henselian valuation tamely branching at $p$.
We split the proof into cases:
\begin{enumerate}
\item \emph{If $Kv = Kv(p)$,}
then we have $\mathcal{O}_v \subseteq \mathcal{O}_{v_K^p}$, so
the canonical $p$-henselian valuation $v_K^p$ is also tamely branching at $p$: 
The fact that $\mathrm{char}(Kv_K^p) \neq p$ is immediate. Furthermore,
we have $v_K^pK = vK/\Delta$ where $\Delta$ is the value group
of the valuation $\bar{v}$ induced by $v$ on $Kv_K^p$. Since we have 
$Kv = Kv(p)$, we also get $Kv_K^p = Kv_K^p(p)$ by the definition of the 
canonical henselian valuation. Thus, $\Delta$ is $p$-divisible and --
as $vK$ is not $p$-divisible -- 
$v_K^p$ is not $p$-divisible. Moreover, we get $[vK:pvK] = [v_K^pK:pv_K^pK]$.
Thus, in case $[v_K^pK:pv_K^pK]=p$, $Kv$ admits a finite separable 
extension of degree divisible by $p^2$, say generated by an irreducible 
polynomial 
$f(X)\in Kv[X]$. Any lift of this polynomial to 
$\mathcal{O}_{\bar{v}}[X] \subseteq Kv_K^p[X]$ is still irreducible and separable
 and thus
also generates a finite separable extension of degree divisible by $p^2$.
Hence, in this case $v_K^p$ is also tamely branching at $p$ as claimed and,
since $v_K^p$ is $\emptyset$-definable, we have found a $\emptyset$-definable
$p$-henselian valuation on $K$.
\item \emph{$Kv\neq Kv(p)$ and $\mathrm{char}(Kv_K^p) \neq p$:} 
Then, we have 
$\mathcal{O}_{v_K^p} \subseteq \mathcal{O}_v$ and thus $v_K^pK$ is not 
$p$-divisible. 
If $Kv_K^p\neq Kv_K^p(p)$ holds, then $v_K^p$ is again tamely branching at $p$.
Now assume that we have $Kv_K^p=Kv_K^p(p)$ and $[v_K^pK:pv_K^pK]=p$. 
Then either $vK$ is $p$-divisible or the value group of the valuation 
$\bar{v}_K^p$ induced by $v_K^p$ on $Kv$ is $p$-divisible.
The first case cannot happen since $v$ is tamely branching at $p$ by assumption.
Hence, assume that $v_K^p$ induces a valuation with $p$-divisible
value group on $Kv$.
As $\bar{v}_K^p$
is $p$-henselian of residue characteristic different to $p$ and its
residue field admits no Galois extension of degree $p$, this implies
$Kv=Kv(p)$. Thus, we get $v=v_K^p$ and so in either case $v_K^p$ is a
$\emptyset$-definable $p$-henselian valuation tamely branching at $p$.
\item \emph{$Kv\neq Kv(p)$ and $\mathrm{char}(Kv_K^p) = p$:}
Define $v_K^pK=:\Gamma$ and $v(p)=:\gamma$.
Consider the convex subgroup
$$\Delta_\gamma := \Set{ \delta \in \Gamma | [0,p\cdot|\delta|] 
\subseteq [0,p\cdot \gamma]+p\Gamma}$$
of $\Gamma$ as in Lemma \ref{Dg}.
We claim that $D_\gamma \neq \Gamma$ holds. 
Let $\langle \gamma \rangle$ be the convex subgroup of $\Gamma$ 
generated by $\gamma$. Then, for any $\delta \in \Delta_\gamma$ there
is some $\beta \in \Gamma$ with
$$\delta-p\cdot \beta \in \langle \gamma \rangle.$$
Note that $v(p)=0$ holds, so
$vK$ is a quotient of $\Gamma/\langle \gamma \rangle$. As $vK$ is not 
$p$-divisible, $\Gamma/\langle \gamma \rangle$ is not $p$-divisible. Hence,
we get $$\Delta_\gamma \subseteq \langle \gamma \rangle+ p\Gamma
\subsetneq \Gamma.$$
This proves the claim.

By the claim, there is
a non-trivial 
$\emptyset$-definable coarsening $u$ of $v_K^p$ on 
$K$ with value group $uK=\Gamma/\Delta_\gamma$. 
Lemma \ref{Dg} implies $uK\neq p\cdot uK$ and $\mathrm{char}(Ku)\neq p$. 
In particular, $u$ is a proper coarsening of 
$v_K^p$. Therefore, $u$ 
is $p$-henselian and $Ku\neq Ku(p)$ holds. Hence, $u$ is an 
$\emptyset$-definable $p$-henselian valuation tamely branching at $p$.
\end{enumerate}
\end{proof} 

We now give a Galois-theoretic consequence of the above.
Together with Theorem \ref{EKN}, Proposition \ref{ptame} yields:
\begin{Cor} \label{cor:tame}
Let $p$ be a prime and $K$ a field with $\mathrm{char}(K)\neq p$ and 
$\zeta_p \in K$. 
Take some $L\equiv K$. Then, if $G_K(p)$ has a non-trivial normal
abelian subgroup, so does $G_L(p)$.
\end{Cor}
\begin{proof} Assume $L \equiv K$. By \cite[Lemma 17]{CDM}, this implies 
$G_K \equiv G_L$ in the language of inverse systems introduced in 
\cite[\S 2]{CDM}. Moreover, as the maximal pro-$p$ quotient of a profinite group
is interpretable in this language, we even get $G_K(p) \equiv G_L(p)$.
If $G_K(p) \cong \mathbb{Z}_p$ or $p=2$ and either
$G_K(p) \cong \mathbb{Z}/2\mathbb{Z}$ or $\mathbb{Z}_2 \ltimes
\mathbb{Z}/2\mathbb{Z}$ holds, then -- as all these groups are small --
we conclude $G_K(p) \cong G_L(p)$ 
(\cite[Proposition 27]{CDM}). 
Hence, $G_L(p)$ also has a non-trivial abelian
normal subgroup.

Otherwise, $K$ admits a $p$-henselian valuation tamely branching at $p$
by Theorem \ref{EKN}. Thus, $K$ admits a $\emptyset$-definable such valuation
by Proposition \ref{ptame}, so $L$ also admits a
$p$-henselian valuation tamely branching at $p$. Using Theorem \ref{EKN}
once more, we get that $G_L(p)$ has a non-trivial normal abelian
subgroup.
\end{proof}

\subsection{Defining tamely branching henselian valuations} 

The main motivation to study henselian valuations tamely branching at some 
prime $p$ is the fact 
that they are encoded in the absolute Galois group of the field.

\begin{Thm}[{\cite[Theorem 1]{Koe03}}, see also {\cite[Theorem 5.4.3]{EP05}}]
A field $K$ admits a henselian valuation, tamely branching at some prime $p$ iff
$G_K$ has a non-procyclic Sylow subgroup $P \not\cong \mathbb{Z}_2 \rtimes
\mathbb{Z}/2\mathbb{Z}$ with a non-trivial abelian normal closed subgroup
$N$ of $P$. \label{tame}
\end{Thm}

The absolute Galois group of a field $K$ is encoded up to elementary
equivalence (when considered in a language for profinite groups)
in the theory of $K$. We now ask whether the Galois-theoretic condition
occuring
in the Theorem above is elementary:
\begin{Q} Let $K$ be a field and $p$ a prime such that $K$ admits a henselian 
valuation tamely branching at $p$. Take $L \equiv K$. Does $L$
admit a henselian valuation tamely branching at $p$, i.e.\;is there a
non-procyclic Sylow subgroup 
$P_L \not\cong \mathbb{Z}_2 \rtimes
\mathbb{Z}/2\mathbb{Z}$ of $G_L$ \label{Q} 
admitting a non-trivial abelian normal closed subgroup?
\end{Q}

For a field $K$ with small absolute Galois group,
the answer to Question \ref{Q} is `yes': 
In case $G_K$ is small, we have $G_K \cong G_L$
(as profinite groups) for any $L$ with $L \equiv K$ 
(\cite[Proposition 4.2]{Kl74}).
The next Proposition gives an alternative way to see
this:
\begin{Prop} \label{smalltame}
Let $K$ be a field and $p$ a prime. Assume that $G_K$ is small.
If $K$ admits a henselian valuation $v$ tamely
branching at $p$, then there is some $\emptyset$-definable coarsening of 
$v$ which tamely branches at $p$.
\end{Prop}
\begin{proof} We may assume that $K$ cotains a primitive $p$th root of
unity $\zeta_p$: As in previous proofs, $K(\zeta_p)$ is an
$\emptyset$-interpretable Galois extension of $K$. Let $u$ be a 
henselian valuation on $K$ and $u'$ its 
unique extension to $K(\zeta_p)$.
 Now, as the index $[K(\zeta_p):K]$ is prime to $p$, $u$ is tamely 
branching at $p$ if and only if $u'$ is tamely branching at $p$.
Thus, any parameter-free definition of a coarsening of $v'$
on $K(\zeta_p)$
which tamely branches at $p$ induces an $\emptyset$-definable  
such coarsening of $v$ on $K$.

Let $v$ be a henselian valuation on $K$ which tamely branches at $p$.
Then $vK$ is not $p$-divisible and -- as $G_K$ is small --
not $p$-antiregular (see the proof of Corollary \ref{Cor1}).
Thus, some non-trivial coarsening $w$ of $v$ is
$\emptyset$-definable by Proposition \ref{antireg}.
Following the proof of Proposition \ref{antireg}, we get that 
either the value group of $w$
is $p$-regular and non-$p$-divisible or $w$ is the finest coarsening of
$v$ with $p$-divisible value group.

We claim that there is an $\emptyset$-definable coarsening $w'$
of $v$ with non-$p$-divisible value group.
Assume first that $wK$ is $p$-regular and non-$p$-divisible:
Then, we can choose $w'=w$.
Assume now that $w$ is the finest coarsening of
$v$ with $p$-divisible value group. Then, $v$ induces a henselian
valuation $\bar{v}$ on $Kw$ such that its value group $\bar{v}(Kw)$
is not $p$-divisible and has no
non-trivial $p$-divisible quotient. 
In particular, $\bar{v}(Kw)$ is either $p$-antiregular or has finite 
(archimedean) rank.
As $G_{Kw}$ is a quotient of $G_K$ 
(\cite[Lemma 5.2.6]{EP05}), $G_{Kw}$ is also small and hence $Kw$ admits no 
henselian valuation with non-$p$-divisible $p$-antiregular value group 
(see again the proof of Corollary \ref{Cor1}).
Thus, $\bar{v}$ is a henselian valuation of finite (archimedean) rank on $Kw$
such that no non-trivial coarsening of it has $p$-divisible value group.
In particular, $\bar{v}$ has
 a (henselian) 
rank-$1$ coarsening $u$ such that the value group $u(Kw)$ is not 
$p$-divisible. Hence, by \cite[Lemma 3.6]{Koe04} (alternatively Theorem
\ref{Hong}), $u$ is $\emptyset$-definable on $Kw$. 
Thus, the composition $w'=u \circ w$ is a $\emptyset$-definable henselian
valuation on $K$ with non-$p$-divisible value group.
This proves the claim.

We have now found a $\emptyset$-definable coarsening
$w'$ of $v$ such that $w'K$ is not $p$-divisible. We
claim that $w'$ is tamely branching at $p$. Since $w'$ coarsens $v$, we have 
$\mathrm{char}(Kw')\neq p$. Assume $p^2 \nmid G_{Kw'}$. Then, as 
$G_{Kv}$ is a quotient of $G_{Kw'}$ (\cite[Lemma 5.2.6]{EP05}), we also get
$p^2 \nmid G_{Kv}$. As $v$ is tamely branching at $p$, we
get $[vK:pvK]\neq p$. 
Furthermore, $p^2 \nmid G_{Kw'}$ implies that
all valuations on $Kw'$ have $p$-divisible value group. Thus,
we get $[w'K:pw'K]=[vK:pvK]\neq p$. Therefore, $w'$ is
tamely branching at $p$.
\end{proof}

In general, Question \ref{Q} has however a negative answer:
\begin{Ex} \label{PZ2}
Consider the field $K$ as constructed in Example \ref{PZ}, so $K$
is elementarily equivalent to a henselian field but not henselian,
$v_K^2$ is $\emptyset$-definable and its value group $v_K^2K$ is
$p$-antiregular value group for all primes $p$. 

By \cite[Lemma 3.4]{PZ78}, there exists some elementary extension $L \succ K$
such that $L$ is henselian. We now show that the canonical
 henselian valuation $v_L$
on $L$ is tamely branching at all primes $p$.

Note that 
the restriction of $v_L$
to $K$ is henselian and thus trivial. In particular, we get 
$\mathrm{char}(Lv_L)=0$.
Furthermore, $v_L$ is comparable to $v_L^p$.
Since $v_L^p$ and $v_K^p$ are definable by the same 
formula and $p$-antiregularity is an elementary property of an ordered
abelian group, $v_L^pL$ is $p$-antiregular. 
Thus, $v_LL$ is not $p$-divisible and also $p$-antiregular. 
By Lemma \ref{infty}, we have $[v_LL:pv_LL]=\infty$.

Overall, we get that $v_L$ is tamely branching at any prime $p$. In particular,
$L$ admits no $\emptyset$-definable non-trivial henselian valuation.
Proposition \ref{tb} below shows that $L$ admits nonetheless for every prime
$p$ a parameter-definable henselian valuation tamely branching at $p$.
\end{Ex}

We immediately get the following
\begin{Cor} \label{cor:nottame}
Admitting a henselian valuation tamely branching at $p$ is not an elementary 
property, i.e., there are fields $K \prec L$ such that
$G_L$ has a non-procyclic Sylow subgroup 
$P_L \not\cong \mathbb{Z}_2 \rtimes
\mathbb{Z}/2\mathbb{Z}$ 
admitting a non-trivial abelian normal closed subgroup, but
$G_K$ does not.
\end{Cor}

As a consequence, not every field which admits a henselian valuation 
tamely branching at $p$ admits a $\emptyset$-definable such.
The next proposition shows that, nevertheless, there is always a definable 
such:
\begin{Prop} Let $K$ be a field and $p$ a prime.
Assume $K$ admits a henselian valuation $v$ tamely branching at $p$.
Then $K$ admits a definable such (using at most $1$ parameter). \label{tb}
\end{Prop}
\begin{proof} Like in the proof of Proposition
\ref{smalltame}, we may assume $\zeta_p \in K$.
We split the proof into two cases:
\begin{enumerate}
\item If $Kv_K=Kv_K(p)$, then we have $v_K \subseteq v_K^p$.
Proposition \ref{ptame} shows that there is an $\emptyset$-definable 
$p$-henselian valuation $w$ which coarsens $v_K^p$ and which tamely 
branches at $p$. As $w$ is a coarsening of $v_K$, this gives
an $\emptyset$-definable henselian valuation tamely branching at $p$.
\item If $Kv_K\neq Kv_K(p)$, then we have $v_K^p \subseteq v_K \subseteq v$.
Define $\Gamma = v_K^pK$. For any $\gamma \in vK$ let
$\langle \gamma \rangle$ be the convex subgroup generated by $\gamma$ in
$\Gamma$. We consider once more the convex subgroup
$$\Delta_\gamma = \Set{\delta \in \Gamma | [0,p\cdot|\delta|] 
\subseteq [0,p\cdot \gamma]+p\Gamma}$$
of $\Gamma$ as in Lemma \ref{Dg}. Note that -- as in the proof of 
Proposition \ref{nondiv} --
$\Gamma=\Delta_\gamma$ implies that the
quotient $\Gamma/\langle \gamma \rangle$ is $p$-divisible.
Thus, if there is some $\gamma \in vK$ such that 
$\Gamma/\langle \gamma \rangle$ is
not $p$-divisible, then we get a definable coarsening $u$ of $v$
with $uK=\Gamma/\Delta_\gamma$ which tamely branches at $p$.

On the other hand, if $\Gamma/\langle \gamma \rangle$ is
$p$-divisible for all $\gamma \in vK$, then
$vK/\langle \gamma \rangle$ is also
$p$-divisible for all $\gamma \in vK$. This implies that $vK/\tilde{\Delta}$
is $p$-divisible for all convex subgroups $\tilde{\Delta}\leq vK$.
Thus, $vK$ is $p$-regular but not $p$-divisible and thus $\emptyset$-definable 
by Theorem \ref{Hong}.
\end{enumerate}
\end{proof}

\section*{Acknowledgements} The authors would like
to thank Will Anscombe, Raf Cluckers and Immanuel Halupczok for helpful
discussions and comments.

\bibliographystyle{alpha}
\bibliography{franzi}
\end{document}